\documentclass[a4paper,fleqn]{cas-sc}
\usepackage{amsmath}
\usepackage{amssymb}
\usepackage{mathtools}
\usepackage{amsfonts}
\usepackage{graphicx}
\usepackage{nccmath}
\usepackage{amsthm}
\usepackage{hyperref}
\usepackage{amscd}

\usepackage[numbers]{natbib}
\def\tsc#1{\csdef{#1}{\textsc{\lowercase{#1}}\xspace}}
\tsc{WGM}
\tsc{QE}
\tsc{EP}
\tsc{PMS}
\tsc{BEC}
\tsc{DE}
\newtheorem{theorem}{Theorem}[section]
\newtheorem{lemma}[theorem]{Lemma}

\newtheorem{prop}[theorem]{Proposition}
\newtheorem{remark}[theorem]{Remark}
\newtheorem{definition}[theorem]{Definition}

\hypersetup{pdfauthor={Name}}

\newcommand{\GL}{\text{GL}}
\newcommand{\SL}{\text{SL}}

\newcommand{\Um}{\text{Um}}
\newcommand{\E}{\text{E}}
\newcommand{\Sp}{\text{Sp}}
\newcommand{\W}{\text{W}}
\newcommand{\Cm}{\text{Cm}}
\newcommand{\ESp}{\text{ESp}}
\newcommand{\MSE}{\text{MSE}}

\begin{document}
\let\WriteBookmarks\relax
\def\floatpagepagefraction{1}
\def\textpagefraction{.001}

\shorttitle{}

\shortauthors{}
% Gopal sharma et~al.

\title [mode = title]{Symplectic completion over smooth affine algebras }

\author[1]{\color{black} Gopal Sharma}
\ead{d22036@students.iitmandi.ac.in}
\author[1]{\color{black} Sampat Sharma}
\ead{sampat@iitmandi.ac.in}

%//////////////////////////////////////////////////////
%////////////////////////////////
\affiliation[1]{organization={SMSS},
    addressline={IIT Mandi}, 
    city={Mandi},
    postcode={175075}, 
    state={Himachal Pradesh},
    country={India}}


\begin{abstract}   
   In this article, we prove the following results:\\
\noindent \text{(1).} Let $R$ be a smooth affine algebra of dimension $3$ over an algebraically closed field $K$ with $3!\in K$, then we show that $\Um_4(R)=e_1\Sp_4(R)$ and $\Um_4(R[X])=e_1\Sp_4(R[X])$.\\

\noindent \text{(2).}  We also show that if $R$ is a smooth affine algebra of dimension $4$ over an algebraically closed field $K$ with $4!\in K$, and assume that $\W_E(R)$ is divisible, then $\Um_3(R)=e_1\SL_3(R)$. As a consequence it is shown that if $R$ is a smooth affine algebra of dimension $4$ over an algebraically closed field $K$ with $4!\in K$, and assume that $\W_E(R)$ is divisible, then $\Um_4(R)=e_1\Sp_4(R)$.\\

\noindent \text{(3).} We show that if $R$ is a local ring of dimension $3$ with $\frac{1}{3!}\in R$. Then $\Um_4(R[X])=e_1\Sp_4(R[X])$.\\

\noindent \text{(4).} We also show that if $R=\oplus_{i\geq 0}R_i$ is a graded ring over a local ring of dimension $3$ with $\frac{1}{3!}\in R$. Then $\Um_4(R)=e_1\Sp_4(R)$.
\end{abstract}

% Keywords
\begin{keywords}
Unimodular rows \sep Completion of unimodular rows \sep Symplectic group \sep Projective module \sep Stably free module 
\end{keywords}

\maketitle

\section{Introduction}
Throughout this article, $R$ will denote a commutative noetherian ring with $1\neq 0$. Let $R$ be a smooth affine algebra of dimension $d\geq 3$ over a field $K$. It follows from \cite[Chapter IV, Theorem 3.4]{bass2006algebraic}, that $\E_n(R)$ acts transitively on $\Um_n(R)$ if $n\geq d+2$. In \cite[Theorem~1]{suslin1977stably}, Suslin showed that $\SL_n(R)$ acts transitively on $\Um_n(R)$ if $K$ is algebraically closed field and $n=d+1$. In \cite{suslin1982cancellation} it is shown that $\SL_n(R)$ acts transitively on $\Um_n(R)$ if cohomological dimension of $K$ satisfies $c.d.(K)\leq 1,~d!\in K^*$ and $n=d+1$. Furthermore, it was proven in \cite{fasel2012stably} that $\SL_n(R)$ acts transitively on $\Um_n(R)$ if $K$ is algebraically closed field, $(d-1)!\in K^*$ and $n=d$. \\

Let $n\geq 1$ and let $R$ be an affine algebra of dimension $d>1$ over a field $K$. It follows from \cite[Lemma 5.5]{suslin1976serre}, that $\Sp_{2n}(R)$ acts transitively on $\Um_{2n}(R)$ if $2n\geq d+2$. It was proven in \cite{basu2011some} that if $K$ is an infinite field of $c.d.(K)\leq 1$ and $4|(d-1)$, then $\Sp_{d+1}(R)$ acts transitively on $\Um_{d+1}(R)$. Furthermore, if $R$ is a non-singular affine algebra of dimension $d$ with $\frac{1}{d!}\in K$ over an algebraically closed field $K$ such that $4|(d-2)$, then it was shown in \cite{gupta2015optimal} that $\Sp_{d}(R)$ acts transitively on $\Um_d(R)$. \\

By using methods of $\mathbb{A}^1$-homotopy theory, Syed in \cite{syed2024symplectic} proved the following:

\begin{theorem}\label{thm:theorem1.1}
	Let $R$ be a smooth affine algebra of odd dimension $d\geq 3$ over a field $K$ such that $c.d.(K)\leq 1$ and $d!\in K^*$; if $d+1\equiv 0~\text{mod}~4$, furthermore assume that $K$ is perfect. Let $\psi$ be an invertible alternating matrix of rank $d+1$. Then $\Sp(\psi)$ acts transitively on $\Um_{d+1}(R)$.
	
\end{theorem}

\begin{theorem}\label{thm:theorem1.2}
	Let $R$ be a smooth affine algebra of even dimension $d\geq 4$ over an algebraically closed field $K$ with $d!\in K^*$. Let $\psi$ be an invertible alternating matrix of rank $d$. Then $\Sp(\psi)$ acts transitively on $\Um_d(R)$.
\end{theorem}

In this article, we investigate the transitivity of the group $\Sp_4(R)$ on $\Um_4(R)$ by classical methods. In particular we prove the following results:

\begin{theorem}\label{thm:theorem1.3}
	Let $R$ be a smooth affine algebra of dimension $3$ over an algebraically closed field $K$ with $3!\in K^*$. Then 
	\begin{enumerate}[(i)]
		\item $\Um_4(R)=e_1\Sp_4(R)$
		\item $\Um_4(R[X])=e_1\Sp_4(R[X]).$
	\end{enumerate}
	  
\end{theorem}

 \begin{theorem}\label{thm:theorem1.4}
 		Let $R$ be a smooth affine algebra of dimension $4$ over an algebraically closed field $K$. Assume that $4!\in K^*$ and $\W_E(R)$ is divisible. Then 
 		$\Um_3(R)=e_1\SL_3(R)$. As a consequence, $\Um_4(R)=e_1\Sp_4(R)$.
 	
 \end{theorem}

In \cite{rao1988bass} Rao studied the problem of completion of unimodular rows over $R[X]$, where $R$ is a commutative noetherian local ring. Rao showed that if $R$ is a local ring of dimension $d$, $d\geq 2$, $\frac{1}{d!}\in R$, then any unimodular row over $R[X]$ of length $d+1$ can be mapped to a factorial row by elementary transformations. In particular, he proved the following: 
\begin{theorem}\label{thm:rao1}
	Let $R$ be a local ring of dimension $d\geq 2$. If $\frac{1}{d!}\in R$, then every unimodular row of length $d+1$ over $R[X]$ can be elementarily transformed to a unimodular row of the form $(w_0^{d!},w_1,...,w_d)$. In particular, $\Um_{d+1}(R[X])=e_1\SL_{d+1}(R[X])$.
\end{theorem}
In this article, we prove symplectic completion over $R[X]$, where $R$ is a local ring of dimension $3$. We prove the following result:
\begin{theorem}\label{thm:theorem1.5}
	Let $R$ be a commutative noetherian local ring of dimension $3$ with $\frac{1}{3!}\in R$. Then $\Um_4(R[X])=e_1\Sp_4(R[X])$.
\end{theorem}

We also prove the relative version of Theorem~\ref{thm:theorem1.5}.

\begin{theorem}\label{thm:theorem1.6}
	Let $R$ be a commutative noetherian local ring of dimension $3$ with $\frac{1}{3!}\in R$. Let $I\subset R$ be an ideal. Then $\Um_4(R[X],I[X])=e_1\Sp_4(R[X],I[X])$.
	\end{theorem}

By using Swan-Weibel's homotopy trick (\cite[Appendix~3]{lam2006serre}), we establish the graded analogue of Theorem~\ref{thm:theorem1.5}. In particular, we prove the following result:

\begin{theorem}\label{thm:theorem1.7}
	Let $A=\oplus_{i\geq 0}A_i$ be a graded ring of dimension $3$ with  $\frac{1}{3!}\in A$. Assume $A_0$ is a local ring. Then $\Um_{4}(A)=e_1\Sp_{4}(A)$.
	
\end{theorem}

\section{Preliminaries} 

\begin{definition}\label{def:definition2.1}
	A row $v=(v_0,v_1,...,v_n)\in R^{n+1}$ is said to be unimodular if there is a row $u=(u_0,u_1,...,u_n)\in R^{n+1}$ with $<v,u>=\sum\limits_{i=0}^{n}v_iu_i=1$ and $\Um_{n+1}(R)$ will denote the set of unimodular rows over $R$ of length $n+1$.
\end{definition}

Let $\GL_{n+1}(R)$ denote the group of all invertible $(n+1)\times (n+1)$ matrices over $R$ and $\SL_{n+1}(R)$ denote the subgroup of $\GL_{n+1}(R)$ consisting of all matrices of determinant 1. The right-action of $\GL_{n+1}(R)$ (and hence of any subgroup) on $\Um_{n+1}(R)$ is by matrix multiplication. Let $G$ be a subgroup of $\GL_{n+1}(R)$, we write $u\sim_G v$ if there exists $\mu \in G$ such that $v=u\mu$. We abbreviate the notations $u\sim_{\GL_{n+1}(R)} v$ to $u\sim v$ and $u\sim_{\SL_{n+1}(R)} v$ to $u\sim_{SL} v$. The group of elementary matrices is a subgroup of $\GL_{n+1}(R)$, denoted by $\E_{n+1}(R)$. The group $\E_{n+1}(R)$ is generated by the matrices of the form $e_{ij}(\lambda)=I_{n+1}+\lambda \E_{ij}$, where $\lambda\in R$, $i\neq j$, $1\leq i,j\leq n+1$, $\E_{ij}\in M_{n+1}(R)$ whose $ij^{th}$ entry is $1$ and all other entries are zero. We abbreviate the notation $u\sim_{\E_{n+1}(R)} v$ to $u\sim_E v$. The matrices $e_{ij}(\lambda)$ will be referred to as $elementary~matrices$.    \\

We shall regard $\GL_{n+1}(R)$ as a subgroup of $\GL_{n+m+1}(R)$ by the map $$\alpha \mapsto \begin{pmatrix}
	\alpha & 0 \\
	0 & I_m
\end{pmatrix}, \alpha\in \GL_{n+1}(R).$$
Let 
	\begin{align*}
  \GL(R)&=\bigcup\limits_{n=0}^{\infty}\GL_{n+1}(R)\\
  \SL(R)&=\bigcup\limits_{n=0}^{\infty}\SL_{n+1}(R)\\
  \E(R)&=\bigcup\limits_{n=0}^{\infty}\E_{n+1}(R).\\
\end{align*}

The group $\E_{n+1}(R)$ is called the elementary linear group and it acts on the rows of length $n+1$ by right multiplication. Moreover, this action takes unimodular rows to unimodular rows:  \\

$\frac{\Um_{n+1}(R)}{\E_{n+1}(R)}$ will denote set of orbits of this action; and we shall denote by $[v]$ the equivalence class of a row $v$ under this equivalence relation.
In \cite[Theorem $3.6$]{van1983group}, W. van der Kallen derives an abelian group structure on $\frac{\Um_{d+1}(R)}{\E_{d+1}(R)}$ when $R$ is of dimension $d\geq 2$.
Let $\Cm_{n+1}(R)$ denote the subset of $\Um_{n+1}(R)$ consisting of the (completable) unimodular rows which can be completed to a matrix of determinant $1$ i.e. $\Cm_{n+1}(R)=\{v\in \Um_{n+1}(R) : v=e_1\alpha\ ~\text{for some}~\alpha \in \SL_{n+1}(R)\}$, where $e_1=(1,0,\cdot\cdot\cdot,0)$ is the first standard basis row vector. Now let $n\geq 2$. In \cite[Corollary 2.34]{rao2017homotopy}, it is shown that the orbit set of completable unimodular rows over $R[X]$, when $R$ is a local ring, modulo the elementary action has an abelian group structure under matrix multiplication i.e. orbit set $\frac{\Cm_{n+1}(R[X])}{\E_{n+1}(R[X])}$ has abelian group structure for $n\geq 2$. To describe this group structure, recall that for a unimodular row $v\in\Cm_{n+1}(R[X])$ there exists $\alpha\in \SL_{n+1}(R[X])$ such that $v=e_1\alpha$. Given two classes $[v],~[w]\in \frac{\Cm_{n+1}(R[X])}{\E_{n+1}(R[X])}$ with representatives $v=e_1\alpha$ and $w=e_1\beta$, we define their product by 
\begin{ceqn}
	\begin{align*}
		[v]\cdot [w]= [e_1\alpha\beta].
	\end{align*}
\end{ceqn}
This operation is well-defined (see \cite[Theorem 2.33]{rao2017homotopy}). The identity is given by $[e_1]$, the inverse of $[e_1\alpha]$ is $[e_1\alpha^{-1}]$, and associativity follows from matrix multiplication. A nontrivial fact, is that this group law is commutative, follows from the fact, $e_1\alpha\beta\sim_E e_1\beta\alpha$, as commutator $[\alpha, \beta]\in \E_{n+1}(R[X])$ (see \cite[Corollary~ 2.20]{rao2017homotopy}).
\begin{lemma}\label{lemma:lemma2.1}\cite[Theorem 2.33]{rao2017homotopy}
	Let $R=\oplus_{d\geq 0}R_d$ be a graded ring with augmentation ideal $R_+=\oplus_{d\geq 1}R_d$. Then for $n\geq 3$, $\frac{\Cm_n(R,R_+)}{\E_n(R,R_+)}$ has an abelian group structure under matrix multiplication. In particular, for $n\geq 3$, the first row map 
	\begin{ceqn}
	\begin{align*}
		\SL_n(R,R_+)\rightarrow & \frac{\Cm_n(R,R_+)}{\E_n(R,R_+)}\\
		\sigma\mapsto & [e_1\sigma]
	\end{align*}
		\end{ceqn}
	is a group homomorphism.
\end{lemma}

We recall the well-known "Swan-Weibel's homotopy trick", which is the main ingredient to handle the graded case.

\begin{definition}\label{def:definition2.2'}
	Let $R=\oplus_{d\geq 0}R_d$ be a graded ring. For $r=r_0+r_1+\cdot\cdot\cdot\in R$ with $r_i\in R_i$, define the Swan-Weibel's homotopy map 
	\begin{ceqn}
		\begin{align*}
			\epsilon : R\rightarrow R[X],~~\epsilon (r)(X)=r_0+r_1X+r_2X^2+\cdot\cdot\cdot+r_iX^i+\cdot\cdot\cdot .
		\end{align*}
	\end{ceqn} 
	
	Thus $\epsilon (r)(0)=r_0\in R_0$ and $\epsilon (r)(1)=r$. 
\end{definition}
In particular, $\epsilon$ gives a homotopy inside $R[X]$ connecting the degree-zero component $R_0$ with the whole ring $R$.
For any $a\in R_0$, one can also evaluate $\epsilon(R)(X)$ at $X=a$, but the important cases are $a=0$ (projection to $R_0$) and $a=1$ (recovery of $r$).\\
The map $\epsilon$ induces a group homomorphism at the $\GL_n(R)$ level for every $n\geq 1$, i.e. for $\alpha \in \GL_n(R)$ we get a map \\ 
	   	$$\epsilon : \GL_n(R)\rightarrow \GL_n(R[X])$$ defined by \\
		$$\alpha=\alpha_0\oplus\alpha_1\oplus\alpha_2\oplus\cdot\cdot\cdot\mapsto \epsilon (\alpha)(X)= \alpha_0\oplus\alpha_1X\oplus\alpha_2X^2\oplus\cdot\cdot\cdot,$$
		where $\alpha_i\in M_n(R_i)$. As above for $a\in R_0$, we evaluate $\epsilon (\alpha)(X)=\alpha_0\oplus\alpha_1X\oplus\alpha_2X^2\oplus\cdot\cdot\cdot$ at $X=a$.

\begin{remark}\label{rmk:remark2.1'}
	The Swan-Weibel's homotopy map provides a homotopy between the projection $R\rightarrow R_0$ (via evaluation at $X=0$) and the identity on $R$ (via evaluation at $X=1$). This is the key idea behind the "Swan-Weibel's homotopy trick": one can "deform" algebraic objects over the whole graded ring $R$ to objects to its degree-zero part $R_0$, and then back. This homotopy trick is particularly useful for transferring problems about $\GL_n(R)$ or projective modules over $R$ to the simpler setting of $R_0$.
	
\end{remark}
\begin{definition}\label{def:definition2.2}
	A matrix $V\in M_n(R)$ is said to be alternating if there exists $W\in M_n(R)$ such that $V=W-W^T$, i.e. $V$ is a skew-symmetric and its diagonal entries are zero. 
\end{definition}

\noindent \textbf{The group $W_E(R)$}: If $V\in M_{2n}(R)$ is alternating then $\text{det}(V)=(\text{Pf}(V))^2$ where Pf is a polynomial (called the Pfaffian) in the matrix elements with coefficients $\pm 1$.\\

If $\phi \in M_r(R) , \psi \in M_s(R)$ are matrices then $\phi \perp \psi$ denotes the matrix $\begin{pmatrix}
	\phi & 0 \\
	0 & \psi 
\end{pmatrix}
\in M_{r+s}(R)$. Let $\psi_1$ denote the matrix 
$\begin{pmatrix}
	0 & 1 \\
	-1 & 0 
	
\end{pmatrix}
\in \E_2(\mathbb{Z})$, and $\psi_r$ is inductively defined by $\psi_r=\psi_{r-1}\perp \psi_1\in \E_{2r}(\mathbb{Z})$, for $r\geq 2$.
\\

For any $\phi \in M_{2r}(R)$ and any alternating matrix $V\in M_{2r}(R)$, we have $\text{Pf}(\phi^TV\phi)=\text{Pf}(V)\text{det}(\phi)$. For alternating matrices $V,\, W$ it is easy to check that $\text{Pf}(V\perp W)=(\text{Pf}(V))(\text{Pf}(W))$.
Note that we need to fix a sign in the choice of $\text{Pf}$; so insist $\text{Pf}(\psi_r)=1$ for all $r$. \\
Two matrices $\alpha \in M_{2r}(R),\, \beta \in M_{2s}(R)$ are said to be equivalent w.r.t. $\E(R)$ if there exists a matrix $\epsilon \in \SL_{2(r+s+l)}(R)\cap \E(R)$, such that $\alpha \perp \psi_{s+l}=\epsilon^T(\beta \perp \psi_{r+l})\epsilon$, for some $l$. Denote this by $\alpha \sim_E \beta$. Then $\sim_E$ is an equivalence relation; denote by $[\alpha]$ the orbit of $\alpha$ under this relation.
It is easy to see \cite{suslin1976serre} that $\perp$ induces the structure of an abelian group on the set of all equivalence classes of alternating matrices with Pfaffian $1$; this group is called "elementary symplectic Witt group" and is denoted by \textbf{$W_E(R)$}. \\

\noindent \textbf{The Vaserstein rule}: \\

Let $R$ be a commutative ring and $v=(v_0,v_1,v_2),\, w=(w_0,w_1,w_2)\in \Um_3(R)$ such that $v\cdot w^T=1$. In \cite{suslin1976serre}, Vaserstein associated an alternating matrix $V(v,w)$ to the pair $v,\, w$: \\
$$V(v,w)=\begin{bmatrix}
	0 & v_0 & v_1 & v_2 & \\
	-v_0 & 0 & w_2 & -w_1 \\
	-v_1 & -w_2 & 0 & w_0 \\
	-v_2 & w_1 & -w_0 & 0 
\end{bmatrix} \in \SL_4(R)$$ with $\text{Pf}(V(v,w))=1$. Associating $V(v,w)$ to $v$ yields a well defined map $V:\frac{\Um_3(R)}{\E_3(R)}\rightarrow \W_E(R)$ called the Vaserstein symbol (see \cite[Theorem~5.2($a_1$)]{suslin1976serre}). In \cite[Lemma~5.1]{suslin1976serre}, Vaserstein proved that the element of $\W_E(R)$ defined by $V(v,w)$ depends only on $v$, and not on the choice of $w$. Furthermore, in \cite[Theorem~5.2]{suslin1976serre}, the Vaserstein rule is established, which is as follows: \\

\begin{lemma}\label{lemma:gsharma2.1} (\cite[Theorem~$5.2(a_2$)]{suslin1976serre}) Let $R$ be a commutative ring  and $v_1=(a_0,a_1,a_2),\, v_2=(a_0,b_1,b_2)$ be two unimodular rows. Suppose $a_0a_0'+a_1a_1'+a_2a_2'=1$ and let $$v_3=\left(a_0,(b_1,b_2)\begin{pmatrix}
		a_1 & a_2 \\
		-a_2' & a_1'
	\end{pmatrix}\right)\in \Um_3(R).$$ 
	Then for any $w_1,w_2,w_3 \in \Um_3(R)$ such that $v_i\cdot w_i^T=1$ for $i=1,2,3$, we have $$[V(v_1,w_1)]\perp [V(v_2,w_2)]=[V(v_3,w_3)]~ \text{in}~W_E(R).$$
	
\end{lemma}

Let $R$ be a ring and $v=(v_0,v_1,v_2),~v^{(n)}=(v_0^n,v_1,v_2)\in \Um_3(R)$ and let $w,~ w_1\in \Um_3(R)$ such that $v\cdot w^T=v^{(n)}\cdot w_1^T=1$. The following is proved in \cite[Lemma~$7.4$]{fasel2012stably}. 

\begin{lemma}\label{lemma:lemma2.1'}(\cite[Lemma~$7.4$]{fasel2012stably})
	Let $R$ be a ring and $v=(v_0,v_1,v_2)\in \Um_3(R)$. Suppose that $(v_0,v_1,v_2)$ and $(-v_0,v_1,v_2)$ are in the same elementary orbit. Then for any $n\in \mathbb{N}$,
	\begin{ceqn}
	\begin{align*}
		[V(v,w)]^n=[V(v^{(n)},w_1)]
	\end{align*}
\end{ceqn}
in $\W_E{R}$.
\end{lemma}

\begin{definition}\label{def:definition2.3}(Symplectic group $\Sp_{2n}(R)$). We define the symplectic group as $\Sp_{2n}(R)=\{\alpha\in \GL_{2n}(R):\alpha^t\psi_n\alpha=\psi_n\}$. We define the relative symplectic group as $\Sp_{2n}(R,I)=\{\alpha\in \Sp_{2n}(R):\alpha\equiv I_{2n}~(\text{mod~I}~)\}$ for any ideal $I$ of $R$.
	
\end{definition}

Let $\sigma\in S_{2n}$ denote the permutation of the natural numbers given by $\sigma(2i)=2i-1$ and $\sigma(2i-1)=2i;~i=1,2,\cdot\cdot\cdot,n$.

\begin{definition}\label{def:definition2.4}(Elementary symplectic group $\ESp_{2n}(R)$). We define for $z\in R,~1\leq i\neq j\leq 2n$,\\
	\begin{ceqn}
	\begin{equation}
	se_{ij}(z)=	\begin{cases}
			1_{2n}+z\E_{ij}   &  \text{if $i=\sigma(j)$;}\\
			1_{2n}+z\E_{ij}-(-1)^{i+j}z\E_{\sigma(j)\sigma(i)}  & \text{if $i\neq \sigma(j)$.}
		\end{cases}
	\end{equation}
\end{ceqn}
	It is easy to see that all these generators belong to $\Sp_{2n}(R)$. We call them elementary symplectic matrices over $R$, and the subgroup of $\Sp_{2n}(R)$ generated by them is called the elementary symplectic group $\ESp_{2n}(R)$. Similarly, the subgroup generated by $se_{ij}(z),~z\in I$ is denoted by $\ESp_{2n}(I)$. The group $\ESp_{2n}(R,I)$ is defined to be the smallest normal subgroup of $\ESp_{2n}(R)$ containing $\ESp_{2n}(I)$.
	
\end{definition}
 To prove the relative version of Theorem~\ref{thm:theorem1.5}, we show how the relative case $(R\neq I)$ can be reduced to the absolute case (see \cite[3.19]{van1983group}).
\begin{definition}\label{def:definition2.5}(The excision ring)
	If $I$ is an ideal of $R$, one constructs the ring $\mathbb{Z}\oplus I$ with multiplication defined by 
	\begin{ceqn}
		\begin{align*}
			(n,i)(m,j)=(nm,nj+mi+ij)
		\end{align*}
	\end{ceqn}
	for $m,~n\in \mathbb{Z},~i,j\in I$. If dimension of the ring is $d\geq 1$, then the maximal spectrum of $\mathbb{Z}\oplus I$ is the union of finitely many subspaces of dimension at most $d$.
	
\end{definition}

There is a natural ring homomorphism $f:\mathbb{Z}\oplus I\rightarrow R$ given by $(m,i)\mapsto m+i\in R$. Let $v=(1+i_1,i_2,\cdot\cdot\cdot,i_n)\in \Um_n(R,I)$ where $i_j$'s are in $I$. Then we shall say $\tilde{v}=(\tilde{1}+\tilde{i_1},\tilde{i_2},\cdot\cdot\cdot,\tilde{i_n})\in \Um_n(\mathbb{Z}\oplus I,0\oplus I)$ for $\tilde{1}=(1,0),~\tilde{i_j}=(0,i_j)$ to be a lift of $v$. Clearly, $f$ sends $\tilde{v}$ to $v$.\\

We shall define $\MSE_n(R)=\frac{\Um_n(R)}{\E_n(R)}$ and likewise $\MSE_n(R,I)=\frac{\Um_n(R,I)}{\E_n(R,I)}$. We recall the Excision theorem (see \cite[Theorem~3.21]{van1983group}).

\begin{theorem}\label{thm:theorem2.1'}(Excision theorem)
	Let $n\geq 3$ be an integer and $I$ be an ideal of a commutative ring $R$. Then the natural maps $F:\MSE_n(\mathbb{Z}\oplus I,0\oplus I)\rightarrow \MSE_n(R,I)$ defined by $[(a_i)]\mapsto [(f(a_i))]$ and $G:\MSE_n(\mathbb{Z}\oplus I,0\oplus I)\rightarrow \MSE_n(\mathbb{Z}\oplus I)$ defined by $[(a_i)]\mapsto [(a_i)]$ are bijections.
	
\end{theorem}

The following is proved in \cite[Proposition~3.1]{keshari2009cancellation}.

\begin{prop}\label{prop:proposition2.1}
	Let $R$ be a ring of dimension $d$ and $I$ be a finitely generated ideal of $R$. \\
	Consider the Cartesian square 
	\begin{ceqn}
		\begin{align*}
		\begin{CD}
			C @>>> R \\
			@VVV    @VVV \\
			R @>>> R/I
		\end{CD}
		\end{align*}
	\end{ceqn}
	Then, $C$ is a finitely generated algebra of dimension $d$ over $R$ and integral over $R$. In fact, $C\simeq R\oplus I$ with the coordinate wise addition and the multiplication defined by $(a,i)(b,j)=(ab,aj+ib+ij),~\tilde{1}=(1,0)$ being the identity in $C$. In particular, if $R$ is an affine algebra of dimension $d$ over a field $K$, then $C\simeq R\oplus I$ is also an affine algebra of dimension $d$ over $K$.
\end{prop}

\begin{remark}\label{rmk:remark2.1}
	We shall call $C$ as the excision algebra of $R$ w.r.t. the ideal $I$. There is a natural ring homomorphism $g:R\oplus I\rightarrow R$ given by $(x,i)\mapsto x+i\in R$.
	
\end{remark}

Let $v=(1+i_1,i_2,\cdot\cdot\cdot,i_n)\in \Um_n(R,I)$ where $i_j$'s are in $I$. Then we shall call $\tilde{v}=(\tilde{1}+\tilde{i_1},\tilde{i_2},\cdot\cdot\cdot,\tilde{i_n})\in \Um_n(R\oplus I,0\oplus I)$ for $\tilde{1}=(1,0),~\tilde{i_j}=(0,i_j)$ to be a lift of $v$. Clearly, $g$ sends $\tilde{v}$ to $v$.

\begin{lemma}\label{lemma:lemma2.2'}(\cite[Lemma~4.3]{gupta2014nice})
	Let $(R,m)$ be a local ring with maximal ideal $m$. Then the excision ring $R\oplus I$ with respect to a proper ideal $I$ in $R$ is also a local ring with maximal ideal $m\oplus I$.
	
\end{lemma}

\begin{lemma}\label{lemma:lemma2.2}(\cite[Lemma 5.5]{suslin1976serre})
	For any natural number $n\geq 2$ and any alternating matrix $\psi$ from $\GL_{2n}(R)$, we have
	\begin{ceqn}
	\begin{align*}
			e_1(\E_{2n}(R))=e_1(\E_{2n}(R)\cap \Sp_{\psi}(R)),
	\end{align*}
 	\end{ceqn}
	where
	\begin{ceqn}
	\begin{align*}
	  \Sp_{\psi}(R)=\{\alpha\in \SL_{2n}(R):\alpha^t\psi \alpha=\psi\}.
		\end{align*}
		\end{ceqn}
\end{lemma}

The following is proved in \cite[Theorem~3.9]{gupta2015optimal}.
\begin{lemma}\label{lemma:lemma2.3'}
	Let $v\in\Um_{2n}(R)$. Then $v\E_{2n}(R)=v\ESp_{2n}(R)$.

\end{lemma}
 \begin{lemma}\label{lemma:lemma2.4}(\cite[Lemma~2.14]{chattopadhyay2011elementary})
 	Let $R$ be a commutative ring and let $\epsilon \in \E_{2n}(R),~n\geq 2$. Then there exists $\rho \in \E_{2n-1}(R)$ such that $\epsilon(1\perp \rho)\in \ESp_{2n}(R)$.
 \end{lemma}

Next we note the results which will be useful for the proof of the main results over the polynomial ring.

\begin{lemma}\label{lemma:lemma2.5}(\cite[Corollary 2.5]{rao1988bass})
	Let $R$ be a ring of dimension $d$. If $\frac{1}{d!}\in R$, then every $v\in \Um_{d+1}(R[X])$ is extended from $R$, i.e. $v\sim_\SL v(0)$.
\end{lemma}

\begin{lemma}\label{lemma:lemma2.6}(\cite[Corollary 3.3]{rao1991completing})
	Let $R$ be a ring of dimension $3$ with $\frac{1}{3!}\in R$. Then any stably extended projective module over $R[X_1,\cdot\cdot\cdot,X_n]$ is extended from $R$.
\end{lemma}

We note the following result of Swan and Towber in \cite[Theorem~2.1]{swan1975class}.
\begin{theorem}\label{thm:theorem3.1'}
	If $p\alpha+q\beta+r\gamma=1$, then 
	\begin{ceqn}
		\begin{align*}
			\begin{vmatrix}
				\alpha^2 & \beta & \gamma \\
				\beta+r\alpha & -r^2+pr\beta & -p+qr-pq\beta \\
				\gamma-q\alpha & p+qr+pr\gamma & -q^2-pq\gamma
			\end{vmatrix}=1.
		\end{align*}
	\end{ceqn}
	
\end{theorem}
The following is the generalized result of Theorem~\ref{thm:theorem3.1'} by Suslin in \cite[Theorem~2]{suslin1977stably}.
\begin{theorem}\label{thm:theorem3.1}
	Let $R$ be a commutative noetherian ring with unity and $(a_0,a_1,\cdot\cdot\cdot,a_r)\in \Um_{r+1}(R)$. Let $n_0,n_1,\cdot\cdot\cdot,n_r$ be positive integers. Suppose that $\prod\limits_{i=0}^rn_i$ is divisible by $r!$. Then there exists a matrix $\alpha\in \SL_{r+1}(R)$ such that $e_1\alpha=(a_0^{n_0},a_1^{n_1},\cdot\cdot\cdot,a_r^{n_r})$.
\end{theorem}

By using Theorem~\ref{thm:theorem3.1}, Suslin proved the following result in \cite[Theorem~1]{suslin1977stably}.

\begin{theorem}\label{thm:theorem3.2}
	Let $R$ be an affine algebra over an algebraically closed field $K$. Then every stably free $R$-module of rank dim $R$ is free.
\end{theorem}

Next we note an implicit result of Fasel, Rao and Swan in \cite[Theorem 7.5]{fasel2012stably}.
\begin{theorem}\label{thm:theorem3.3}
	Let $R$ be a smooth affine algebra of dimension $d\geq 3$ over an algebraically closed field $K$ and $(gcd(d-1)!,char(K))=1$. Let $v=(v_1,\cdot\cdot\cdot,v_d)\in \Um_d(R)$. Then there exists $\epsilon \in \E_d(R)$ such that
	\begin{ceqn} 
	\begin{align*}
		v\epsilon=(w_1,w_2,\cdot\cdot\cdot,w_d^{(d-1)!})
	\end{align*}
\end{ceqn}
	for some $(w_1,w_2,\cdot\cdot\cdot,w_d)\in \Um_d(R)$.
	
\end{theorem}

%*****************************************************************************************************************************************************************************************************************************************************************************************************************************

\section{The Main Theorems} 

In this section, we prove the main results Theorem~\ref{thm:theorem1.3} and Theorem~\ref{thm:theorem1.4}. 

\begin{theorem}\label{thm:theorem3.4}
	Let $R$ be a smooth affine algebra of dimension $3$ over an algebraically closed field $K$ with $3!\in K^*$. Then 
	\begin{enumerate}[(i)]
		\item $\Um_4(R)=e_1\Sp_4(R)$
		\item $\Um_4(R[X])=e_1\Sp_4(R[X]).$
	\end{enumerate}
	
\end{theorem}

\begin{proof}
	\begin{enumerate}[(i)]
		\item Let $v\in \Um_4(R)$. In view of Theorem \ref{thm:theorem3.2}, there exists $\sigma \in \SL_4(R)$ such that 
		\begin{ceqn}
	\begin{align*}
		v=e_1\sigma.
	\end{align*}
\end{ceqn}
	Note that $\sigma^t\psi_2\sigma$ is an alternating matrix, therefore we have 
	\begin{ceqn}
	\begin{align*}
		\sigma^t\psi_2\sigma=V(v',w')
	\end{align*}
\end{ceqn}
	for some $v',w'\in \Um_3(R)$ such that $v'\cdot (w')^t=1$. 
	Since $\Um_3(R)=e_1\SL_3(R)$ by Theorem~\ref{thm:theorem3.3}, there exists $\rho\in \SL_3(R)$ such that 
    \begin{ceqn}
	\begin{align*}
		v'=e_1\rho.
	\end{align*}
\end{ceqn}
	 By \cite[Lemma 5.1,~Theorem 5.2]{suslin1976serre}, in $W_E(R)$ we have
	\begin{ceqn} 
	\begin{align*}
		[V(v',w')]=&[V(e_1\rho,e_1(\rho^{-1})^t)]\\
		=&[(1\perp \rho)^t\psi_2(1\perp \rho)].
	\end{align*}
\end{ceqn}
	Thus there exists $\epsilon \in \E_4(R)$ such that
	\begin{ceqn}
	\begin{align*}
		V(v',w')=\epsilon^t(1\perp \rho)^t\psi_2(1\perp \rho)\epsilon.
	\end{align*}
\end{ceqn}
	In view of Lemma \ref{lemma:lemma2.4}, there exists $\epsilon_1\in \E_3(R)$ such that $\epsilon=\delta_1(1\perp \epsilon_1)^{-1}$ for some $\delta_1\in \ESp_4(R)$.
	Therefore, 
	\begin{ceqn}
	\begin{align*}
        \sigma^t\psi_2\sigma= (1\perp \epsilon_1^{-1})^t\delta_1^t(1\perp \rho)^t\psi_2(1\perp \rho)\delta_1(1\perp \epsilon_1^{-1}).
        	\end{align*}
    \end{ceqn}
     \text{Thus},
    \begin{ceqn}
    	\begin{align*}
        (1\perp \rho^{-1})^t(\delta_1^{-1})^t(1\perp \epsilon_1)^t\sigma^t\psi_2\sigma(1\perp \epsilon_1)\delta_1^{-1}(1\perp \rho^{-1})=\psi_2.
	\end{align*}
\end{ceqn}
	Thus $\sigma(1\perp \epsilon_1)\delta_1^{-1}(1\perp \rho^{-1})\in \Sp_4(R)$. Let $\delta=\sigma(1\perp \epsilon_1)\delta_1^{-1}(1\perp \rho^{-1})$. Therefore,
	\begin{ceqn}
	\begin{align*}
		\sigma=\delta(1\perp \rho)\delta_1(1\perp \epsilon_1^{-1}).
	\end{align*}
\end{ceqn} 
	In view of Lemma \ref{lemma:lemma2.1}, we have 
	\begin{ceqn}
	\begin{align*}
		[e_1\sigma]&=[e_1(\delta(1\perp \rho)\delta_1(1\perp \epsilon_1^{-1}))]\\
		           &=[e_1\delta]\ast [e_1(1\perp \rho)]\ast [e_1\delta_1]\ast [e_1(1\perp \epsilon_1^{-1})].
	\end{align*}
\end{ceqn}
	Thus $[e_1\sigma]=[e_1\delta\delta_1]$ in $\frac{\Cm_4(R)}{\E_4(R)}$. Therefore, there exists $\epsilon'\in\E_4(R)$ such that 
	\begin{ceqn}
	\begin{align*}
		e_1\sigma=(e_1\delta\delta_1)\epsilon'.
	\end{align*}
\end{ceqn}
	In view of Lemma \ref{lemma:lemma2.3'}, there exists $\epsilon_1'\in \ESp_4(R)$ such that 
	\begin{ceqn}
	\begin{align*}
		(e_1\delta\delta_1)\epsilon'=(e_1\delta\delta_1)\epsilon_1'.
	\end{align*} 
\end{ceqn}
	Thus $e_1\sigma=e_1(\delta\delta_1\epsilon_1')$. Note that $\delta\delta_1\epsilon_1'\in \Sp_4(R)$. Therefore, $v=e_1\sigma =e_1(\delta\delta_1\epsilon_1')\in e_1\Sp_4(R)$.\\
		\item First, we show the following:
		\begin{align*}
			&(a)~\Um_4(R[X])=e_1\SL_4(R[X])\\
			&(b)~\Um_3(R[X])=e_1\SL_3(R[X]).
		\end{align*}
		(a)$~$ Let $v(X)\in \Um_4(R[X])$. By Lemma \ref{lemma:lemma2.5}, $v(X)\sim_\SL v(0)$. Now by Theorem \ref{thm:theorem3.2}, $v(0)\sim_\SL e_1$. Thus $v(X)\sim_\SL e_1$.\\
		
		(b)$~$  Let $v(X)\in \Um_3(R[X])$. By Lemma \ref{lemma:lemma2.6}, $v(X)\sim_\SL v(0)$. Now by Theorem \ref{thm:theorem3.3}, $v(0)\sim_\SL e_1$. Thus $v(X)\sim_\SL e_1$.
	\end{enumerate}
The proof of part (ii) proceeds similarly to that of part (i).
\end{proof}

 \begin{theorem}\label{thm:theorem3.5}
	Let $R$ be a smooth affine algebra of dimension $4$ over an algebraically closed field $K$. Assume that $4!\in K^*$ and $\W_E(R)$ is divisible. Then 
	$\Um_3(R)=e_1\SL_3(R)$. As a consequence $\Um_4(R)=e_1\Sp_4(R)$.
	
\end{theorem}

\begin{proof}
	First, we will show that $\Um_3(R)=e_1\SL_3(R)$. Let $v=(v_0,v_1,v_2)\in \Um_3(R)$. By \cite[Corollary 7.9]{fasel2012stably}, the set $\frac{\Um_3(R)}{\E_3(R)}$ is in bijection with $\W_E(R)$ and is thus endowed with the structure of an abelian group. By assumption $\W_E(R)$ is divisible and $\frac{\Um_3(R)}{\E_3(R)}$ is in bijection with $\W_E(R)$, $\frac{\Um_3(R)}{\E_3(R)}$ is a divisible group prime to the characteristic of $K$. Therefore, there exists a unimodular row $(u_0,u_1,u_2)\in \Um_3(R)$ such that 
	\begin{ceqn}
	\begin{align*}
		[(v_0,v_1,v_2)]=[(u_0,u_1,u_2)]^2
	\end{align*}
\end{ceqn}
	in $\frac{\Um_3(R)}{\E_3(R)}$. Since $-1$ is a square in $K$, Lemma \ref{lemma:lemma2.1'}, shows that $[(v_0,v_1,v_2)]^n=[(v_0^n,v_1,v_2)]$ in $\frac{\Um_3(R)}{\E_3(R)}$ for any $n\in \mathbb{N}$. Thus $[(u_0,u_1,u_2)]^2=[(u_0^2,u_1,u_2)]$. By Theorem~\ref{thm:theorem3.1'}, there exists $\alpha\in\SL_3(R)$ such that
	\begin{ceqn}
		\begin{align*}
			(u_0^2,u_1,u_2)=e_1\alpha
		\end{align*}
	\end{ceqn}  
	
	Thus, $\Um_3(R)=e_1\SL_3(R)$. Note that the implicit result of Fasel-Rao-Swan (cf.~Theorem~\ref{thm:theorem3.3}) asserts that $\Um_4(R)=e_1\SL_4(R)$. Now the proof of the last statement of the theorem is analogous to the proof of the part (i) of the Theorem \ref{thm:theorem3.4}.

\end{proof}

\section{Symplectic completion over polynomial rings}
 
 In \cite{rao1988bass}, Rao proved that if $R$ is a local ring of dimension $d\geq 2$, $v\in \Um_{d+1}(R[X])$ and $\frac{1}{d!}\in R$, then $v\sim_\E(w_0^{d!},w_1,\cdot\cdot\cdot,w_d)\sim_\SL e_1$ for some $(w_0,\cdot\cdot\cdot,w_d)\in \Um_{d+1}(R[X])$. In this section we study the symplectic completion over polynomial extension of local ring of dimension $3$.
 
 \begin{theorem}\label{thm:theorem4.1}(\cite[Theorem 2.4]{rao1988bass})
 	Let $R$ be a local ring of dimension $d\geq 2$. Let $v\in \Um_{d+1}(R[X])$. If $\frac{1}{d!}\in R$, then $v\sim_\E(w_0^{d!},w_1,\cdot\cdot\cdot,w_d)\sim_\SL e_1$ for some $(w_0,\cdot\cdot\cdot,w_d)\in \Um_{d+1}(R[X])$.
 \end{theorem} 

\begin{theorem}\label{thm:theorem4.2}(\cite[Theorem 3.1]{rao1991completing})
	Let $R$ be a local ring of dimension $3$ with $\frac{1}{2}\in R$. Let $v=(v_0,v_1,v_2)\in \Um_3(R[X])$. Then $v$ can be completed to an invertible matrix i.e. $v\sim_\SL e_1$.
	
\end{theorem}

Now we are ready to prove the Theorem \ref{thm:theorem1.5}:

\begin{theorem}\label{thm:theorem4.3}
	Let $R$ be a commutative noetherian local ring of dimension $3$ with $\frac{1}{3!}\in R$. Then $\Um_4(R[X])=e_1\Sp_4(R[X])$.
\end{theorem}

\begin{proof}
	The proof is analogous to the proof of the Theorem \ref{thm:theorem3.4}. For the convenience of the reader, we give a precise proof with all the necessary adjustments to the proof of Theorem \ref{thm:theorem3.4}:\\
		
		Let $v\in \Um_4(R[X])$. In view of Theorem \ref{thm:theorem4.1}, there exists $\sigma \in \SL_4(R[X])$ such that 
		\begin{ceqn}
		\begin{align*}
			v=e_1\sigma.
		\end{align*}
			\end{ceqn}
		Note that $\sigma^t\psi_2\sigma$ is an alternating matrix, therefore we have
		\begin{ceqn} 
		\begin{align*}
			\sigma^t\psi_2\sigma=V(v',w')
		\end{align*}
		\end{ceqn}
		for some $v',w'\in \Um_3(R[X])$ such that $v'\cdot (w')^t=1$. 
		Now by Theorem \ref{thm:theorem4.2}, $\Um_3(R[X])=e_1\SL_3(R[X])$, there exists $\rho\in \SL_3(R[X])$ such that 
		\begin{ceqn}
		\begin{align*}
			v'=e_1\rho.
		\end{align*}
			\end{ceqn}
	By \cite[Lemma 5.1,~Theorem 5.2]{suslin1976serre}, in $W_E(R[X])$  we have
	\begin{ceqn} 
		\begin{align*}
			[V(v',w')]=&[V(e_1\rho,e_1(\rho^{-1})^t)]\\
			=&[(1\perp \rho)^t\psi_2(1\perp \rho)].
		\end{align*}
			\end{ceqn}
		Thus there exists $\epsilon \in \E_4(R[X])$ such that
		\begin{ceqn}
		\begin{align*}
			V(v',w')=\epsilon^t(1\perp \rho)^t\psi_2(1\perp \rho)\epsilon.
		\end{align*}
	\end{ceqn}
		In view of Lemma \ref{lemma:lemma2.4}, there exists $\epsilon_1\in \E_3(R[X])$ such that $\epsilon=\delta_1(1\perp \epsilon_1)^{-1}$ for some $\delta_1\in \ESp_4(R[X])$.
		Therefore, 
		\begin{ceqn}
		\begin{align*}
			\sigma^t\psi_2\sigma= (1\perp \epsilon_1^{-1})^t\delta_1^t(1\perp \rho)^t\psi_2(1\perp \rho)\delta_1(1\perp \epsilon_1^{-1}).
				\end{align*}
		\end{ceqn}
		Thus,
		\begin{ceqn}
			\begin{align*}
			(1\perp \rho^{-1})^t(\delta_1^{-1})^t(1\perp \epsilon_1)^t\sigma^t\psi_2\sigma(1\perp \epsilon_1)\delta_1^{-1}(1\perp \rho^{-1})=\psi_2.
		\end{align*}
	\end{ceqn}
		Thus $\sigma(1\perp \epsilon_1)\delta_1^{-1}(1\perp \rho^{-1})\in \Sp_4(R[X])$. Let $\delta=\sigma(1\perp \epsilon_1)\delta_1^{-1}(1\perp \rho^{-1})$. Therefore,
		\begin{ceqn} 
		\begin{align*}
			\sigma=\delta(1\perp \rho)\delta_1(1\perp \epsilon_1^{-1}).
		\end{align*}
		\end{ceqn}
		In view of Lemma \ref{lemma:lemma2.1}, we have 
		\begin{ceqn}
		\begin{align*}
			[e_1\sigma]&=[e_1(\delta(1\perp \rho)\delta_1(1\perp \epsilon_1^{-1}))]\\
			&=[e_1\delta]\ast [e_1(1\perp \rho)]\ast [e_1\delta_1]\ast [e_1(1\perp \epsilon_1^{-1})].
		\end{align*}
	\end{ceqn}
		Thus $[e_1\sigma]=[e_1\delta\delta_1]$ in $\frac{\Cm_4(R[X])}{\E_4(R[X])}$. Therefore, there exists $\epsilon'\in\E_4(R[X])$ such that
		\begin{ceqn} 
		\begin{align*}
			e_1\sigma=(e_1\delta\delta_1)\epsilon'.
		\end{align*}
		\end{ceqn}
		In view of Lemma \ref{lemma:lemma2.3'}, there exists $\epsilon_1'\in \ESp_4(R[X])$ such that 
		\begin{ceqn}
		\begin{align*}
			(e_1\delta\delta_1)\epsilon'=(e_1\delta\delta_1)\epsilon_1'.
		\end{align*} 
	\end{ceqn}
		Thus $e_1\sigma=e_1(\delta\delta_1\epsilon_1')$. Note that $\delta\delta_1\epsilon_1'\in \Sp_4(R[X])$. Therefore, $v=e_1\sigma =e_1(\delta\delta_1\epsilon_1')\in e_1\Sp_4(R[X])$.
\end{proof}

Next, we prove the relative version of Theorem~\ref{thm:theorem4.3}.

\begin{theorem}\label{thm:theorem4.4}
	Let $R$ be a commutative noetherian local ring of dimension $3$ with $\frac{1}{3!}\in R$. Let $I\subset R$ be an ideal. Then $\Um_4(R[X],I[X])=e_1\Sp_4(R[X],I[X])$.
\end{theorem}
\begin{proof}
	Let $v(X)\in \Um_4(R[X],I[X])$ and $\tilde{v}(X)\in \Um_4((R\oplus I)[X],(0\oplus I)[X])$ be a lift of $v(X)$. By Lemma~\ref{lemma:lemma2.2'}, $R\oplus I$ is also a local ring of dimension $3$. So by Theorem~\ref{thm:theorem4.3}, there exists $\sigma\in \Sp_4((R\oplus I)[X])$ such that $\tilde{v}(X)\sigma=e_1$. Going modulo $(0\oplus I)[X]$, we have $e_1\bar{\sigma}=e_1,~\bar{\sigma}\in\Sp_4(R[X])$.\\
	
	Therefore,
	\begin{ceqn}
		\begin{align*}
			\tilde{v}(X)\sigma=e_1&=e_1\bar{\sigma}.
		\end{align*}
	\end{ceqn} 
	We can view $\bar{\sigma}$ as an element of  $\Sp_4((R\oplus I)[X])$ via the canonical embedding $R[X] \hookrightarrow (R\oplus I)[X]$. Thus, both   $\sigma$  and $\bar{\sigma}$ live in the same group $\Sp_4((R\oplus I)[X])$, so their product is well-defined. Let $\delta=\sigma\bar{\sigma}^{-1}$. Note that $\delta\in\Sp_4((R\oplus I)[X],(0\oplus I)[X])$ and $\tilde{v}(X)\delta=e_1$. Taking projection onto $R[X]$ we obtain the desired result i.e. $v(X)\in e_1\Sp_4(R[X],I[X])$.
\end{proof}

\section{Symplectic completion over graded rings}
In this section we study the symplectic completion over graded ring. Throughout this section we assume $A$ to be a commutative noetherian graded ring with identity $1\neq 0$.  We require the support of the following results to prove the analogous statement over graded rings.

\begin{prop}\label{prop:proposition5.1}
	Let $A=\oplus_{i\geq 0}A_i$ be a graded ring of dimension $d\geq 2$, with  $\frac{1}{d!}\in A$. Assume $A_0$ is a local ring. Then $\Um_{d+1}(A)=e_1\SL_{d+1}(A)$.  
	
\end{prop}
\begin{proof}
	Let $v\in\Um_{d+1}(A)$. Consider the Swan-Weibel's homotopy map: 
	\begin{ceqn}
		\begin{align*}
			\epsilon : A &\rightarrow A[X]\\
	 a_0+a_1+a_2+\cdot\cdot\cdot &\mapsto a_0+a_1X+a_2X^2+\cdot\cdot\cdot+a_nX^n+\cdot\cdot\cdot
		\end{align*}
	\end{ceqn}
	where $a_i\in A_i$. We denote the polynomial $a_0+a_1X+a_2X^2+\cdot\cdot\cdot+a_nX^n+\cdot\cdot\cdot$ by $\epsilon (a)(X)$ for $a=a_0+a_1+a_2+\cdot\cdot\cdot \in A$. This extends component-wise to vectors, so for $v\in \Um_{d+1}(A)$, we note that $\epsilon (v)(X)\in \Um_{d+1}(A[X])$. By Lemma~\ref{lemma:lemma2.5}, we have 
	\begin{ceqn}
		\begin{align*}
			 \epsilon (v)(X)\sim_\SL \epsilon (v)(0).
		\end{align*}
	\end{ceqn}
	Note that $\epsilon (v)(0)\in \Um_{d+1}(A_0)$ and since $A_0$ is a local ring, 
	\begin{ceqn}
		\begin{align*}
			\epsilon (v)(0)\sim_\SL e_1.
		\end{align*}
	\end{ceqn} Thus, 
	\begin{ceqn}
		\begin{align*}
			\epsilon (v)(X)\sim_\SL e_1.
		\end{align*}
	\end{ceqn}
	Evaluating at $X=1$, we obtain
	\begin{ceqn}
		\begin{align*}
			v=\epsilon (v)(1)\sim_\SL e_1.
		\end{align*}
	\end{ceqn} 
	Hence $\Um_{d+1}(A)=e_1\SL_{d+1}(A)$.  
\end{proof}

\begin{prop}\label{prop:proposition5.2}
		Let $A=\oplus_{i\geq 0}A_i$ be a graded ring of dimension $3$, with  $\frac{1}{3!}\in A$. Assume $A_0$ is a local ring. Then $\Um_{3}(A)=e_1\SL_{3}(A)$.  
\end{prop}
\begin{proof}
		Let $v\in\Um_{3}(A)$. Consider the Swan-Weibel's homotopy map: 
	\begin{ceqn}
		\begin{align*}
			\epsilon : A &\rightarrow A[X]\\
			a_0+a_1+a_2+\cdot\cdot\cdot &\mapsto a_0+a_1X+a_2X^2+\cdot\cdot\cdot+a_nX^n+\cdot\cdot\cdot
		\end{align*}
	\end{ceqn}
	where $a_i\in A_i$. We denote the polynomial $a_0+a_1X+a_2X^2+\cdot\cdot\cdot+a_nX^n+\cdot\cdot\cdot$ by $\epsilon (a)(X)$ for $a=a_0+a_1+a_2+\cdot\cdot\cdot \in A$. This extends component-wise to vectors, so for $v\in \Um_{3}(A)$, we note that $\epsilon (v)(X)\in \Um_{3}(A[X])$. By Lemma~\ref{lemma:lemma2.6}, we have 
	\begin{ceqn}
		\begin{align*}
			\epsilon (v)(X)\sim_\SL \epsilon (v)(0).
		\end{align*}
	\end{ceqn}
	Note that $\epsilon (v)(0)\in \Um_{3}(A_0)$ and since $A_0$ is a local ring, 
	\begin{ceqn}
		\begin{align*}
			\epsilon (v)(0)\sim_\SL e_1.
		\end{align*}
	\end{ceqn} Thus, 
	\begin{ceqn}
		\begin{align*}
			\epsilon (v)(X)\sim_\SL e_1.
		\end{align*}
	\end{ceqn}
	Evaluating at $X=1$, we obtain
	\begin{ceqn}
		\begin{align*}
			v=\epsilon (v)(1)\sim_\SL e_1.
		\end{align*}
	\end{ceqn} 
	Hence $\Um_{3}(A)=e_1\SL_{3}(A)$.  
\end{proof}

\begin{theorem}\label{thm:theorem5.1}
	Let $A=\oplus_{i\geq 0}A_i$ be a graded ring of dimension $3$, with  $\frac{1}{3!}\in A$. Assume $A_0$ is a local ring. Then $\Um_{4}(A)=e_1\Sp_{4}(A)$.
	
\end{theorem}

\begin{proof}
		Let $v\in \Um_4(A)$. In view of Proposition \ref{prop:proposition5.1}, there exists $\sigma \in \SL_4(A)$ such that 
	\begin{ceqn}
		\begin{align*}
			v=e_1\sigma.
		\end{align*}
	\end{ceqn}
	Note that $\sigma^t\psi_2\sigma$ is an alternating matrix, therefore we have
	\begin{ceqn} 
		\begin{align*}
			\sigma^t\psi_2\sigma=V(v',w')
		\end{align*}
	\end{ceqn}
	for some $v',w'\in \Um_3(A)$ such that $v'\cdot (w')^t=1$. 
	Now by Proposition \ref{prop:proposition5.2}, $\Um_3(A)=e_1\SL_3(A)$, there exists $\rho\in \SL_3(A)$ such that 
	\begin{ceqn}
		\begin{align*}
			v'=e_1\rho.
		\end{align*}
	\end{ceqn}
	By \cite[Lemma 5.1,~Theorem 5.2]{suslin1976serre}, in $W_E(A)$  we have
	\begin{ceqn} 
		\begin{align*}
			[V(v',w')]=&[V(e_1\rho,e_1(\rho^{-1})^t)]\\
			=&[(1\perp \rho)^t\psi_2(1\perp \rho)].
		\end{align*}
	\end{ceqn}
	Thus there exists $\epsilon \in \E_4(A)$ such that
	\begin{ceqn}
		\begin{align*}
			V(v',w')=\epsilon^t(1\perp \rho)^t\psi_2(1\perp \rho)\epsilon.
		\end{align*}
	\end{ceqn}
	In view of Lemma \ref{lemma:lemma2.4}, there exists $\epsilon_1\in \E_3(A)$ such that $\epsilon=\delta_1(1\perp \epsilon_1)^{-1}$ for some $\delta_1\in \ESp_4(A)$.
	Therefore, 
	\begin{ceqn}
		\begin{align*}
			\sigma^t\psi_2\sigma= (1\perp \epsilon_1^{-1})^t\delta_1^t(1\perp \rho)^t\psi_2(1\perp \rho)\delta_1(1\perp \epsilon_1^{-1}).
		\end{align*}
	\end{ceqn}
	Thus,
	\begin{ceqn}
		\begin{align*}
			(1\perp \rho^{-1})^t(\delta_1^{-1})^t(1\perp \epsilon_1)^t\sigma^t\psi_2\sigma(1\perp \epsilon_1)\delta_1^{-1}(1\perp \rho^{-1})=\psi_2.
		\end{align*}
	\end{ceqn}
	Thus $\sigma(1\perp \epsilon_1)\delta_1^{-1}(1\perp \rho^{-1})\in \Sp_4(A)$. Let $\delta=\sigma(1\perp \epsilon_1)\delta_1^{-1}(1\perp \rho^{-1})$. Therefore,
	\begin{ceqn} 
		\begin{align*}
			\sigma=\delta(1\perp \rho)\delta_1(1\perp \epsilon_1^{-1}).
		\end{align*}
	\end{ceqn}
	In view of Lemma \ref{lemma:lemma2.1}, we have 
	\begin{ceqn}
		\begin{align*}
			[e_1\sigma]&=[e_1(\delta(1\perp \rho)\delta_1(1\perp \epsilon_1^{-1}))]\\
			&=[e_1\delta]\ast [e_1(1\perp \rho)]\ast [e_1\delta_1]\ast [e_1(1\perp \epsilon_1^{-1})].
		\end{align*}
	\end{ceqn}
	Thus $[e_1\sigma]=[e_1\delta\delta_1]$ in $\frac{\Cm_4(A)}{\E_4(A)}$. Therefore, there exists $\epsilon'\in\E_4(A)$ such that
	\begin{ceqn} 
		\begin{align*}
			e_1\sigma=(e_1\delta\delta_1)\epsilon'.
		\end{align*}
	\end{ceqn}
	In view of Lemma \ref{lemma:lemma2.3'}, there exists $\epsilon_1'\in \ESp_4(A)$ such that 
	\begin{ceqn}
		\begin{align*}
			(e_1\delta\delta_1)\epsilon'=(e_1\delta\delta_1)\epsilon_1'.
		\end{align*} 
	\end{ceqn}
	Thus $e_1\sigma=e_1(\delta\delta_1\epsilon_1')$. Note that $\delta\delta_1\epsilon_1'\in \Sp_4(A)$. Therefore, $v=e_1\sigma =e_1(\delta\delta_1\epsilon_1')\in e_1\Sp_4(A)$.
\end{proof}

\noindent \textbf{Acknowledgments.}\\

 We thank the referee for his thorough reading and valuable suggestions to improve the manuscript.\\
This research was conducted at Indian Institute of Technology (IIT Mandi), and the authors would like to acknowledge the invaluable resources and facilities provided by the institute. The research of "Gopal Sharma" is funded by the University Grants Commission (UGC) with Fellowship No. 221610163684/(CSIRNETJUNE2022). Second author acknowledges DST INSPIRE (DST/INSPIRE/04/2021/002849) and start up research grant (SRG/2022/000056) for their support. The second author also thanks IIT Mandi for their seed grant.\\

% \bibliographystyle{cas-model2-names}
\bibliographystyle{abbrv}

\bibliography{references}


\end{document}